\documentclass[11pt]{amsart}    
\usepackage{fullpage}
\usepackage{amsmath, amssymb, amsthm}
\usepackage{geometry}         
\usepackage{tikz}
\usetikzlibrary{matrix}
\usepackage{amsmath}
\usepackage{amssymb}
\usepackage{tikz}
\usepackage{mathdots}
\usetikzlibrary{matrix,arrows,decorations.pathmorphing}
\usetikzlibrary{shapes.geometric}
\usepackage{tikz-cd}
\usepackage{amsthm}
\usepackage{graphicx}
\usepackage{amssymb}
\usepackage{tikz-cd}
\usetikzlibrary{calc,arrows}
\title{Self-Dual Ramsey Degrees for Trees}
\author{Sebastian Junge}
\date{}
\address{Department of Mathematics, Texas State University, Math and Computer Science Building 470, San Marcos, TX, 78666}
\email{iiw22@txstate.edu}
\newtheorem*{definition}{Definition}

\newtheorem*{lemma*}{Lemma}
\newtheorem{theorem}{Theorem}[section]
\newtheorem{lemma}[theorem]{Lemma}
\newtheorem{proposition}[theorem]{Proposition}

\begin{document}
	\tikzcdset{row sep/my size/.initial=.01 em}
	\begin{abstract} We consider a Ramsey statement for pairs of maps between trees, where one is an embedding as defined by Deuber and the other is a rigid surjection as defined by Solecki. We show that there is no Ramsey Theorem for pairs of maps where the coloring depends on both coordinates. On the other hand, we give a characterization of the Ramsey degrees for such pairs. Furthermore, we show that our theorem on Ramsey Degrees for pairs of maps between trees implies the Ramsey Theorem for pairs of maps between linear orders as proved by Solecki.
	\end{abstract}
	\maketitle
	\section{Introduction}
	Ramsey Theory has a long tradition of considering linear orders, in fact Ramsey's original Theorem can be stated as the following: for any finite linear orders $K$ and $L$ and any $r>0$ there is a finite linear order $M$ so that for any coloring $\chi$ of the increasing injections from $K$ to $M$ there is an increasing injection $j\colon L\to M$ so that $$j\circ \{i\colon K\to L| \ i\text{ is an increasing injection}\}$$
	is $\chi$-monochromatic. Ramsey theorists have also considered maps other than increasing injections between linear orders. The most famous example of this is the Graham-Rotschild Theorem which concerns colorings of certain surjective maps $M\to K$ instead of injective maps $K\to M$. Since this procedure involves switching the order the new Theorem is referred to as the Dual Ramsey Theorem. Traditionally, and more generally, Ramsey theorists have also looked at maps between trees such as in \cite{26},\cite{29},\cite{31},\cite{28},\cite{22},\cite{30}. We will continue this work by looking at pairs of maps between trees. For more context and an introduction to Ramsey Theory the reader may consult Ne\u set\u ril's survey in \cite{32}. 
	
	In \cite{12} Solecki gives a Ramsey statement for certain pairs of maps, called connections, between linear orders. This theorem is known as the Self-Dual Ramsey Theorem. The Self-Dual Ramsey Theorem combines Ramsey's Theorem with the Graham-Rotschild Theorem. In \cite{22} Solecki proves a Dual Ramsey Theorem for trees which also generalizes the Dual Ramsey Theorem for linear orders. A natural question to ask is whether there is a Self-Dual Ramsey Theorem for appropriately defined connections between trees. The goal of this paper is to answer this question. 
	
	First, we show that there is Self-Dual Ramsey Theorem for all maps between trees where the coloring depends on both coordinates. In particular, we show that any definition of connection which includes all rigid surjections (which can be viewed canonically as pairs of maps) and includes a pair which is not a rigid surjection does not fulfill a Ramsey statement.  However, we define connections between trees in a way that follows naturally from the definition of connections between linear orders given in \cite{12}. In Theorem 3.2 we give a characterization of the Ramsey degrees of connections, in particular showing that the Ramsey degrees are finite. Statements about Ramsey degrees measure the distance from the exact Ramsey Theorem. The proof of the upper bounds for the Ramsey degrees will be our main theorem, Theorem 3.4. While Theorem 3.4 is not a Ramsey statement, we show in Section 5 that it implies the Self-Dual Ramsey Theorem for linear orders from \cite{12}. Furthermore this theorem also implies the Dual Ramsey Theorem for trees in \cite{22}, though this is a bit of circular logic since Theorem \ref{2.18} is proved using the main lemma from the proof of the Dual Ramsey Theorem for trees. 
	
	We now describe the organization of this paper. In Section 2, we review the definitions of connections between linear orders from \cite{12} and the definition of rigid surjections for trees from \cite{22}. We also give our notation for Ramsey theory in that section. In Section 3, we state our definition of connections between trees, which follows naturally from the definition of connections between linear orders. In that section we state our main results. In Section 4 we prove the failure of the Self-Dual Ramsey degrees between trees and give lower bounds for the Ramsey degrees for connections. In Section 5, we  prove Theorem \ref{2.18} giving the upper bounds for the Ramsey and show how this theorem implies the Self-Dual Ramsey Theorem for linear ordars from \cite{12}.
	
	This paper is based on work in my Ph.D. thesis.
	
	\section{Background}
	In this section we present the two results that we aimed to combine: the Self-Dual Ramsey Theorem for linear orders and the Dual Ramsey Theorem for trees. We review the theory needed to state these theorems, while showing the connections between them. Then we give the standard notation of Ramsey theory as seen through a categorical viewpoint that we use in this paper. \textbf{From now on, all linear orders and trees are assumed to be finite}.
\subsection{Connections for linear orders}
We state the definitions needed to formulate the Self-Dual Ramsey theorem for linear orders from \cite{12}. We start by defining increasing injections and rigid surjections.
\begin{definition}
	An \textbf{increasing injection} between linear orders $K$ and $L$ is a map $i\colon K\to L$ so that if $x,y\in K$ and $x< y$, then $i(x)<i(y)$.
\end{definition}
These maps are sometimes referred to as strongly increasing maps, but we will use the terminology increasing injection instead. We now define a natural dual for increasing injections. 
\begin{definition}
	Let $K$ and $L$ are linear orders. A map $s\colon L\to K$ is a \textbf{rigid surjection} if the map $i_s\colon K\to L$ defined by $i_s(x)=\textnormal{min}(s^{-1}(x))$ is an increasing injection.
\end{definition}
There are multiple definitions of rigid surjection. We choose this definition, because it generalizes nicely to the definition of rigid surjections for trees. 
This definition of rigid surjections is motivated by the concept of Galois Connections, which Solecki uses in \cite{22} to generalize to  rigid surjections for trees. For ease of notation we let $(g,f)\colon A\leftrightarrows B$ denote that $f\colon A\to B$ and $g\colon B\to A$. 

\begin{definition}
	Let $A,B$ be partial orders. A pair $(f,g)\colon A\to B$ is a Galois connection if $f(g(y))=y$ for all $y\in B$ and $g(f(x))\leq x$ for all $x\in A$.
\end{definition}

Note that if $K$ and $L$ are linear orders and $s\colon L\to K$ is a rigid surjection, then $(s,i_s)$ is a Galois connection. In fact, $s$ is a rigid surjection if and only if there is an increasing injection $i$ so that $(s,i)$ is a Galois connection.

We consider certain pairs of maps between partial orders (in particular linear orders) which we call connections. In doing so we follow \cite{12}.
\begin{definition}
	If $L$ and $K$ are partial orders a \textbf{connection} between $L$ and $K$ is a pair $(s,i)\colon L\leftrightarrows K$ so that for all $x\in K$, $$s(i(x))=x \text{ and for all }y<i(x), \ s(y)\leq x.$$
\end{definition}
Note that if $(s,i)$ is a connection between linear orders then $i$ is an increasing injection and $s$ is a rigid surjection. 

Next we state the Ramsey theorems for these three types of maps. The Ramsey statement for increasing injections is equivalent to Ramsey's Theorem.
\begin{theorem}[Ramsey's Theorem]
	For any linear orders $K$ and $L$ and any $r>0$ there is a linear order $M$ so that for any coloring $\chi$ of the increasing injections from $K$ to $M$ there is an increasing injection $j\colon L\to M$ so that $$j\circ \{i\colon K\to L| \ i\text{ is an increasing injection}\}$$
	is $\chi$-monochromatic.
	\end{theorem}
The Ramsey statement for rigid surjections is known as the Dual Ramsey Theorem. It is a consequence of the Graham-Rothschlid Theorem which has many different formulations.
\begin{theorem}[Graham Rothschild]
For any linear orders $K$ and $L$ and any $r>0$ there is a linear order $M$ so that for any coloring $\chi$ of the rigid surjections from $M$ to $K$ there is a rigid surjection $t\colon T\to S$ so that $$\{i\colon K\to L| \ i\text{ is an increasing injection}\}\circ t$$
is $\chi$-monochromatic.
\end{theorem}
The Ramsey Statement for connections of linear orders the Self-Dual Ramsey Theorem for linear orders in \cite{12}.
\begin{theorem}[\cite{12}, Theorem 2.1]\label{2.9}
	For any linear orders $K$ and $L$ and any $r>0$ there is a linear order $M$ so that for any coloring $\chi$ of the connections from $K$ to $M$ there is a connection $(t,j)\colon M\leftrightarrows L$ so that $$\{(s\circ t, j\circ i)| (s,i)\colon L\leftrightarrows S\text{ is a connection}\}$$
	is $\chi$-monochromatic.
\end{theorem}
The main objective of this paper is to consider the case when $K$ and $L$ are trees instead of linear orders in Theorem \ref{2.9}. To do this we define  versions of increasing injections and rigid surjections for trees following Deuber and Solecki.
\subsection{Trees}
We review the definitions of trees, embedding between trees as defined by Deuber in \cite{26}, and rigid surjections given by Solecki in \cite{22}. 
\begin{definition}
	A \textbf{tree} is a partial ordered set $(T,\sqsubseteq_T)$ which has a least $\sqsubseteq_T$ element called the $\textbf{root}$ and so that for all $x\in T$, the predecessors of $x$ are linearly ordered by $\sqsubseteq_T$.
	
\end{definition}
For all vertices $v,w$ in a tree $T$, there is a largest common predecessor of $v$ and $w$ called the meet of $v$ and $w$ which we denote by $v\wedge_T w$. We will drop the subscript for $\wedge$ and $\sqsubseteq$ when the tree in question is clear. Using the idea of $\wedge$ we  define ordered trees.
\begin{definition}
	If $T$ is a tree and $v\in T$ we let $im_T(v)$ denote all immediate successors of $v$. A tree is \textbf{ordered} if there is a linear order $\leq_v$ on $im(v)$ for all $v\in T$. These orders extend to a linear order on $T$ given by $v\leq_T w$ if $v\sqsubseteq w$ or both $w\not\sqsubseteq v$ and $x\leq_{v\wedge w}y$ where $x,y\in im_T(v\wedge w)$, $x\subseteq_T v$ and $y\subseteq_T w$.
\end{definition}
\textbf{From now on all trees we consider will be ordered}. Note that any linear order $L$ can be viewed as a tree where $\sqsubseteq_L=\leq_L$.  We turn our attention to maps between trees.
\begin{definition}[Deuber]
	If $S,T$ are trees an \textbf{embedding} $i\colon S\to T$ is a function so that the following hold
	\begin{enumerate}
		\item $i$ sends the root of $S$ to the root of $T$.
		\item If $x,y\in S$ and $x<_S y$, then $i(x)<_T i(y)$.
		\item For all $x,y\in S$, $i(x\wedge_S y)=i(x)\wedge_T i(y)$.
	\end{enumerate}
\end{definition}
We will use this definition of embedding to define rigid surjections following \cite{22}. This idea generalizes the definition  of rigid surjections for linear orders in Section 2.1.
\begin{definition}
	Let $S,T$ be trees, we say a map $s\colon T\to S$ is a \textbf{rigid surjection} if there is an embedding $i\colon S\to T$, so that for all $x\in S$ and $y\in T$, $$s(i(x))=x \text{ and } i(s(y))\sqsubseteq_T y$$
\end{definition}
By definition a map $s$ is a rigid surjection if there is an embedding $i$ so that $(s,i)$ is a Galois connection. We remark that if such an embedding exists, then it is unique. Indeed, for any rigid surjection $s$, let $i_s$ be given by $$i_s(x)=\bigwedge s^{-1}(x).$$
Then $i_s$ is the unique embedding so that $(s,i_s)$ is a Galois connection. We think of $i_s$ as the smallest possible embedding so that $s(i_s)$ is the identity.

Now that we have defined rigid surjections between trees, we give the main result of \cite{22} which is a Ramsey statement for these maps.
\begin{theorem}[{\cite[Theorem 2.3]{22}}]\label{2.12}
	Let $S$ and $T$ be trees. Then for any $r>0$, there is a tree $V$, so that for any $r$-coloring $\chi$ of the rigid surjections from $V$ to $S$ there is a rigid surjection $t\colon V\to T$ such that, $$\{s\colon T\to S| \ s\textnormal { is a rigid surjection}\}\circ t$$
	is $\chi$-monochromatic. 
\end{theorem}
\subsection{The Ramsey property}
We give our notation for the standard ideas found in a categorical approach to Ramsey theory.

\begin{definition}
	Fix a category $\mathcal{C}$, if $A,B,C\in \textnormal{Ob}(\mathcal{C})$ and $r>0$, then we say $C$ is a Ramsey witness for $A$ and $B$ (denoted $C\to (B)^{A}_{r}$) if for any $r$-coloring $\chi$ of $\textnormal{Hom}_{\mathcal{C}}(A,C)$ there is  $f\in \textnormal{Hom}_{\mathcal{C}}(B,C)$ so that $f\circ \textnormal{Hom}_{\mathcal{C}}(A,B)$ is $\chi$-monochromatic. 
	
	A category $\mathcal{C}$ has the \textbf{Ramsey property} if for all $A,B\in \textnormal{Ob}(\mathcal{C})$ and for all $r>0$, there is $C\in \textnormal{Ob}(\mathcal{C})$ so that $C\to (B)^{A}_{r}$.
\end{definition}
	In some cases Ramsey theorists study categories which don't have the Ramsey property, but instead we consider different degrees of failure. More specifically we have the following definition. 
	\begin{definition}
	Fix a category $\mathcal{C}$. For any $A,B\in \text{Ob}(\mathcal{C})$ we say that the \textbf{Ramsey degree} of $A$ and $B$ is $k$, denoted by $rd(A,B)=k$, if $k$ is the smallest number where for all $r>0$ there is a $C\in \mathcal{C}$ so that for any coloring  of $\text{Hom}_{\mathcal{C}}(A,C)$ with $r$ colors there is a $f\in \text{Hom}_{\mathcal{C}}(B,C)$ so that the coloring attains at most $k$ colors on $f\circ \text{Hom}_{\mathcal{C}}(A,B)$. \\
	
	For a single $A\in \text{Ob}(\mathcal{C})$ we let the Ramsey Degree of $A$, denoted by $rd(A)$, be given as $$rd(A)=\text{sup}_{B\in\text{Ob}(\mathcal{C})}rd(A,B)$$
This supremum may be $\infty$. 
\end{definition}
In the literature these are referred to as small Ramsey degrees (since big Ramsey degrees occur in infinite Ramsey theory). Since we only cover finite Ramsey Theory, we simply say Ramsey degree.

 Note that a category $\mathcal{C}$ has the Ramsey property if and only if for all $A\in\text{Ob}(\mathcal{C})$, $rd(A)=1$.  Thus $rd(A)>1$ measure how much the Ramsey property fails, with the worst case scenario being that $rd(A)=\infty$. Therefore a common property that Ramsey theorists look for is all Ramsey degrees being finite. 
\section{Connections for trees}
Given the Dual Ramsey Theorem for trees and the Ramsey Theorem for connections of linear orders, we ask whether there is a nice definition of connection for trees that does have the Ramsey property. First we must formalize what we mean by a ``nice" definition. We will consider categories $\mathcal{C}$, whose objects are all trees and whose morphisms are pairs of maps $(s,i)$ between trees, with the standard composition. Furthermore, we wish for all $ \mathcal{C}$ to satisfy the following:
\begin{enumerate}
	\item For all trees $S,T$, and $(s,i)\in \text{Hom}_{\mathcal{C}}(S,T)$, $s\colon T\to S$ is a rigid surjection,  $i\colon S\to T$  is an embedding, and $s\circ i=\text{Id}_{S}$.
	\item For all trees $S$, $T$, and a rigid surjection $s\colon T\to S$, we have that\\ $(s,i_s)\in \text{Hom}_{\mathcal{C}}(S,T)$.
	\item There are trees $S$, $T$, and a $(s,i)\in \text{Hom}_{\mathcal{C}}$ so that $i\neq i_s$.
\end{enumerate}
 The first property is that the pairs are partial inverses. Condition (2) asserts that we include the rigid surjections (since they can be view as pairs $(s,i_s)$). We want (3) to hold since otherwise all maps in the category depend only on the rigid surjections. Surprisingly, the following theorem answers the question in the negative.

\begin{theorem}\label{4.1}
	Let $\mathcal{C}$ be category whose objects are all trees and whose morphisms are pairs of maps $(s,i)$ between trees, with the standard composition and so that the following hold, \begin{enumerate}
		\item For all trees $S,T$, and $(s,i)\in \text{Hom}_{\mathcal{C}}(S,T)$, $s\colon T\to S$ is a rigid surjection,  $i\colon S\to T$  is an embedding, and $s\circ i=\text{Id}_{S}$.
		\item For all trees $S$, $T$, and a rigid surjection $s\colon T\to S$, we have that $(s,i_s)\in \text{Hom}_{\mathcal{C}}(S,T)$.
		\item There are trees $S$, $T$, and a $(s,i)\in \text{Hom}_{\mathcal{C}}$ so that $i\neq i_s$.
	\end{enumerate}Then $\mathcal{C}$ does not have the Ramsey Property.
\end{theorem}

However, we will show that there is a category of connections for trees that does have finite Ramsey degrees. In this section, we will extend the definition of connections for linear orders in a natural manner to trees. We will discuss the implications of this definition. Then we describe the Ramsey degrees for trees in this category of connections.
\begin{definition}
	If $S,T$ are trees then a pair $(s,i)\colon T\leftrightarrows S$ is a \textbf{connection} if the following are true,
	\begin{enumerate}
		\item[(a)] For all $x\in S$, $s(i(x))=x \text{ and for all }y<i(x), \ s(y)\leq x$.
		\item[(b)] $s$ is a rigid surjection.
		\item[(c)] $i$ is an embedding.
	\end{enumerate}
\end{definition}
Note that condition (a) of a connection  is equivalent to $(s,i)$ being a connection between the linear orders $(T,\leq_T)$ and $(S,\leq_S)$. In the case of linear orders, if $(s,i)$ is a connection, then $s$ is a rigid surjection and $i$ is an increasing injection. However, this is not the case for trees, which is why conditions (b) and (c) are needed.

We remark that this definition of connection meets conditions (1)-(3) above. Condition (1) is given and (2) follows from the definition of $i_s$. For condition (3) we need to show that there are connections which are not rigid surjections. We give examples of such connections when proving the lower bounds of the Ramsey degrees for trees.

We define a category of connections between trees as follows. Let $Conn_T$ be the category whose objects are (finite linearly ordered) trees and whose morhpisms are defined by  $\text{Hom}_{Conn_T}(S,T)=\{(s,i)\colon T\leftrightarrows S | (s,i)\text{ is a connection}\}$. Composition is given by $(s,i)\circ (t,j)=(t\circ s,i\circ j)$.

The main result that we will prove in this paper gives a categorization of the Ramsey degrees for connections between trees. 
\begin{theorem}\label{1.2}
	If $S\in \textnormal{Ob}(Conn_T)$ then $$rd(S)= |\mathcal{P}(\{x\in S\colon x\text{ is not the root of }S\text{ and }x\text{ has at most one immediate successor in }S\})|.$$
		Where $\mathcal{P}$ denotes the power set.
\end{theorem}
 We prove this Theorem in two steps. In Section 4 we will prove Theorem \ref{4.1}. Using the same construction, we give a description of the lower bound of the Ramsey degrees.  
 \begin{lemma}\label{2.17}
 	If $S\in \textnormal{Ob}(Conn_T)$ then, $$rd(S)\geq |\mathcal{P}(\{x\in S\colon x\text{ is not the root of }S\text{ and }x\text{ has at most one immediate successor in }S\})|.$$
 
 \end{lemma}

For the other direction of the proof of Theorem \ref{1.2} we  give an upper bound of the Ramsey degrees. 
\begin{theorem}\label{2.18}
	If $S$ is a tree, let $$A=\{x\in S\colon x\text{ is not the root of }S\text{ and }x\text{ has at most one immediate successor in }S\}.$$
	Then for all trees $T$ and any number $r>0$, there is a tree $V$ so that for any coloring 
	$\chi\colon \textnormal{Hom}_{Conn_T}(S,V)\to r$, there is a $(t,j)\in\textnormal{Hom}_{Conn_T}(T,V)$ such that for any connection 
	 $(s,i)\in\textnormal{Hom}_{Conn_T}(S,T)$, $\chi((t,j)\circ (s,i))$ depends only on the set, $$B=\{x\in A\colon i_s(x)\neq i(x)\}.$$
	In particular, $rd(S)\leq |P(A)|$.
\end{theorem}
This theorem is a characterization of the Ramsey degrees $rd(S,T)$ for all $T$. Indeed consider the number $k$ given by, $$k=|\{B\in \mathcal{P}(A)\colon \exists (s,i)\in\text{Hom}_{Conn_T}(S,T)\text{ such that }B=\{x\in A
\colon i_s(x)\neq i(x))\}\}|$$
Then Theorem \ref{2.18} precisely states that $rd(S,T)\leq k$. By definition for any $T$ this $k\leq |\mathcal{P}(A)|$, since it is the cardinality of a subset of $\mathcal{P}(A)$. Hence we see that $rd(S)\leq|\mathcal{P}(A)|$. 

Note that Theorem 3.4 extends the Dual Ramsey Theorem for trees. For any coloring $\chi$ of the rigid surjections between $S$ and $T$, we define a new coloring on the connections by $\chi'(s,i)=\chi(s)$. Theorem 3.4 states that there is a $V$ and $(t,j)\in\textnormal{Hom}_{Conn_T}(T,V)$ such that for any connection 
$(s,i)\in\textnormal{Hom}_{Conn_T}(S,T)$, $\chi'((t,j)\circ (s,i))$ depends only on the set, $$B=\{x\in A\colon i_s(x)\neq i(x)\}.$$
Then for any rigid surjection between $S$ and $T$, $\chi(j\circ s)$ has the same color as $\chi'((t,j)\circ (s,i_s))$. This color is fixed since the set $B$ is empty for all pairs $(s,i_s)$.

In Section 5 we will show that Theorem \ref{2.18} implies the Ramsey Theorem for connections for linear orders. Thus these results do generalize the result for linear orders. Section 5 will also contain more background on general Ramsey Theory and ideas used in the proof of Theorem \ref {2.18}.
\section{Lower bounds on Ramsey Degrees}
In this section we prove Theorem \ref{4.1} which shows that no natural definition of connection has the Ramsey Property and Lemma \ref{2.17} which gives a lower bound on the Ramsey degree of connections. First, we  prove Lemma \ref{2.16} which shows that some maps are invariant under composition. In Proposition \ref{1.1} we construct a tree where this invariant occurs the maximal amount. Then we color maps using this invariant to giving a coloring with no monochromatic subsets. 

\begin{lemma}\label{2.16}
	If $(s,i)\colon T\leftrightarrows S$ is a pair so that $s$ is a rigid surjection, $i$ is an embedding, $s\circ i=\textnormal{Id}_S$, and $x\in S$ has at least two immediate successors then $i(x)=i_s(x)$.
\end{lemma}
Before the proof of Lemma \ref{2.16}, we give the following figure to provide an intuitive idea of how the lemma works.

Let $(s,i)\colon T\leftrightarrows S$, $s$ be a rigid surjection, $i$ an embedding, and $s\circ i=\text{Id}_S$. If $x\in S$ has two distinct immediate successors $x_1,x_2\in S$, then as seen below $i_s(x_1)\wedge i_s(x_2)=i(x_1)\wedge i(x_2)$. Since $i_s$ and $i$ are embeddings, $i_s(x)=i_s(x_1)\wedge i_s(x_2)$ and $i(x)=i(x_1)\wedge i(x_2)$. Hence $i_s(x)=i(x)$.

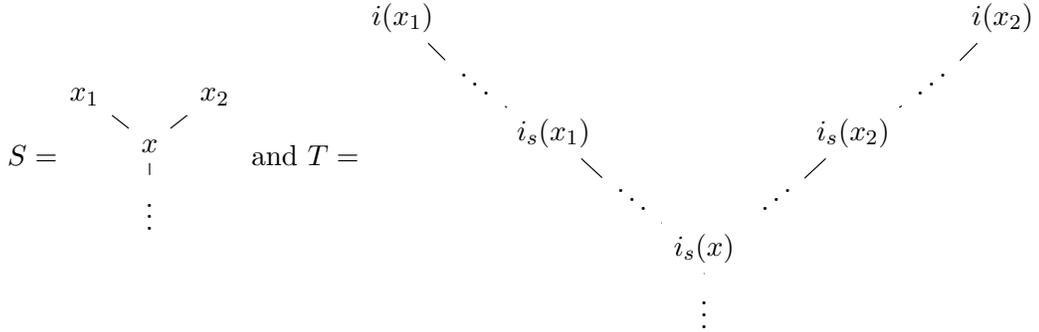
	\begin{center}
		$S=$\begin{tikzcd}[column sep=tiny, row sep=tiny]
			x_1\arrow[dr, dash] & & x_2\arrow[dl, dash]\\
			& x\arrow[d, dash] & \\
			& \vdots & \\
		\end{tikzcd} and $T=$\begin{tikzcd}[column sep=.00001em, row sep= my size]
			i(x_1)\arrow[dr, dash] & & & & & & & & i(x_2)\arrow [dl, dash]\\
			& \ddots \arrow[dr, dash] & & & & & & \iddots\arrow[dl, dash] & \\
			& & i_s(x_1)\arrow[dr, dash] & & & &  i_s(x_2)\arrow[dl, dash] & & \\
			& & & \ddots \arrow[dr, dash] & & \iddots \arrow[dl, dash] & & &\\
			& & & & i_s(x)\arrow[d, dash] & & & &\\
			& & & & \vdots & & & &\\
		\end{tikzcd}
	\end{center}

\begin{proof}[Proof of Lemma \ref{2.16}]
	Let  $(s,i)\colon T\leftrightarrows S$, $s$ be a rigid surjection, $i$ an embedding, $s\circ i=\text{Id}_S$, and  $x\in S$ have two distinct immediate successors $x_1$, $x_2$.  Since $s$ is a rigid surjection, $i_s(x_1)\wedge i_s(x_2)=i_s(x)$, similarly $i(x)=i(x_1)\wedge i(x_2)$. In particular, $i_s(x_1)\not\sqsubseteq i_s(x_2)$ and $i_s(x_2)\not\sqsubseteq i_s(x_1)$ (in which case we say that $i_s(x_1)$ and $i_s(x_2)$ are incompatible). Since $s\circ i=\text{Id}_S$ and $i_s$ is the $\sqsubseteq$ smallest embedding that is a partial inverse of $s$, we have that $i_s(x_1)\sqsubseteq i(x_1)$, $i_s(x_2)\sqsubseteq i(x_2)$, and $i_s(x)\sqsubseteq i(x)$. So, $i(x)$ and $i_s(x_1)$ are compatible since they are predecessors of $i(x_1)$, similarly $i(x)$ and $i_s(x_2)$ are compatible.  Note that $i(x)\sqsubseteq i_s(x_1)$ and $i(x)\sqsubseteq i_s(x_2)$, since otherwise $i_s(x_1)$ and $i_s(x_2)$ would be compatible. Hence, $i(x)\sqsubseteq i_s(x)$ since $i(x)$ is a predecessor of $i_s(x_1)$ and $i_s(x_2)$. As noted above, $i_s(x)\sqsubseteq i(x)$, which completes the proof.
\end{proof}
Lemma \ref{2.16} gives an invariant under composition. Indeed, consider a\ $(s,i)\colon T\rightleftarrows S$ with an $x\in S$ that has $i_s(x)=i(x)\in T$ with multiple immediate successors. Then by Lemma \ref{2.16} any composition with $(s,i)$, $(s\circ t,j\circ i)$ still has $i_{s\circ t}(x)=j\circ i(x)$. We will show in the proof of Theorem \ref{4.1} that $i_s(x)\neq i(x)$ is also an invariant under composition. Thus, to give a precise lower bound on the Ramsey degrees of connections we create a tree where these invariants occur as much as possible.
\begin{proposition}\label{1.1}
	For any tree $S$ let $$A=\{x\in S\colon x\text{ is not the root of }S\text{ and }x\text{ has at most one successor in }S\}.$$ 
	There is a tree $T$ and a rigid surjection $s\colon T\to S$ so that $i_s(x)$ has at least two successor for all $x\in S$ and for any $B\subseteq A$ there is an embedding $i_B\colon S\to T$ for which $(s,i_B)$ is a connection and for all $a\in A$, $i_s(a)\neq i_B(a)$ if and only if $a\in B$.
	\end{proposition}
\begin{proof}
	Fix a tree $S$ and let $$A=\{x\in S\colon x\text{ is not the root of }S\text{ and }x\text{ has at most one successor in }S\}.$$
	We construct a tree $T$ with base set $S\sqcup A\sqcup A$, that is for each $a\in A$ we add two additional vertices $a_1$ and $a_2$. We give the tree relation $\sqsubseteq_T$ which extends $\sqsubseteq_S$ on $S$. Define this relation to be reflexive. If $x\in S$, let $x\sqsubseteq_T  a_1$ and $x\sqsubseteq_T a_2$ if and only if $x\sqsubseteq_S a$. For $y\in S$, let $a_1\sqsubseteq_T y$ if and only if $y\neq a$ and $a\sqsubseteq_S y$. Lastly, for all $b\in A$, $b_1\sqsubseteq_T a_1$ and $ b_1\sqsubseteq_T a_2$ if and only if $a\sqsubseteq_S b$. So each $a\in A$ has two new immediate successors $a_1$ and $a_2$. Furthermore, the immediate successor of $a$ in $S$ is an immediate successor of $a_1$ in $T$ if it exists. Define $\leq_T$ by $a_1\leq_T a_2$ for all $a\in A$ and extend $\leq_S$ otherwise. This is sufficient since all $a\in A$ have at most one successor. 
	The following is an example of a $T$ constructed this way from a tree $S$.
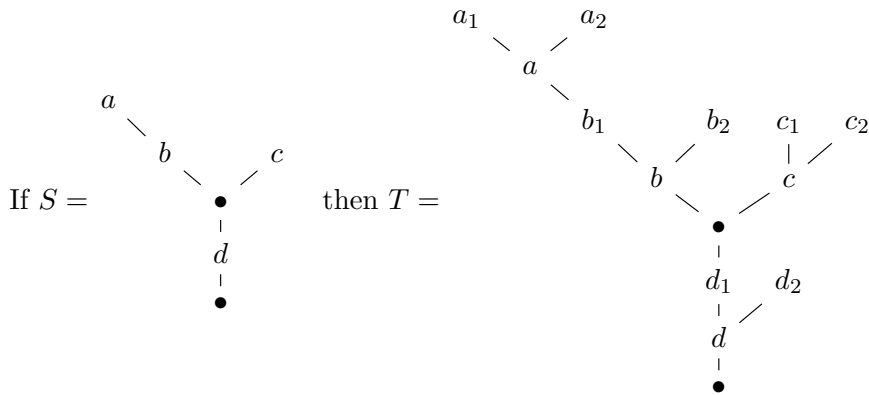
		\begin{center}
			If $S=$\begin{tikzcd}[column sep=tiny, row sep=tiny]
				a\arrow[dr, dash] & & & & \\
				& b\arrow[dr, dash]& & c\arrow[dl, dash]& \\
				& & \bullet\arrow[d, dash] & & \\
				& & d \arrow[d, dash]& & \\
				& & \bullet & &\\
			\end{tikzcd} then  $T=$\begin{tikzcd}[column sep=tiny, row sep=tiny]
				a_1\arrow[dr, dash] & & a_2\arrow[dl, dash] & & & & &\\
				& a\arrow[dr, dash] & & & & & & &\\
				& & b_1\arrow[dr, dash] & & b_2\arrow[dl, dash] & c_1\arrow[d, dash] & c_2\arrow[dl, dash]\\
				& & & b\arrow[dr, dash]& & c\arrow[dl, dash]& \\
				& & & & \bullet\arrow[d, dash] & &  \\
				& & & & d_1\arrow[d, dash] & d_2\arrow[dl, dash]\\
				& & & & d \arrow[d, dash]& & \\
				& & & & \bullet & &\\
			\end{tikzcd}
		\end{center}
	In effort to construct connections between $T$ and $S$ we claim that any map $i\colon S\to T$ which extends the identity on $S\backslash A$ and has $i(a)=a$ or $i(a)=a_1$ for all $a\in A$ is an embedding. Fix such a function $i$. Since the root of $S$ is not in $A$ and  by the definition of $\leq_T$, it is enough to show that $i$ preserves $\wedge$. Suppose that $x,y,z\in S$ and $x\wedge y=z$. We can assume that $z$ is neither $x$ or $y$, since if not $i(z)=i(x)\wedge i(y)$ since $i$ preserves $\sqsubseteq_S$.  Hence $z$ has at least two immediate successors so $z\in S\backslash A$. By the definition of $\sqsubseteq_T$, for any $s\neq x\in S$, $s\sqsubseteq_T i(x)$ if and only if $s\sqsubseteq_S x$ and the same holds for $s\neq y$. So $z$ is the largest common predecessor of $i(x)$ and $i(y)$ in $S$ (seen as a subset of $T$). Since no element of $T\backslash S$ has at least two immediate successors, this means that $z=i(x)\wedge_T i(y)$. Since $z$ is not in $A$ our assumption on $i$ means that $i(z)=z$. Hence the above says that $i(z)=i(x)\wedge_T i(y)$ which proves the claim.
	
 Using the above claim we complete the proof. Let $s\colon T\to S$ be defined by $s(x)=x$ on $S$ and $s(a_1)=s(a_2)=a$. Then $i_s\colon S\to T$ is the identity on $S$ which is an embedding by the claim above. For each $B\subseteq A$, let $i_B\colon S\to T$ be defined by $i(x)=x$ for all $x\in S\backslash B$, and $i(b)=b_1$ for all $b\in B$. By the claim above each $i_B$ is an embedding, also it is easy to check that $(s,i_B)$ is a connection. 
\end{proof}
The failure of the Ramsey property follows rather directly from Lemma \ref{2.16} and Proposition \ref{1.1}. We start by proving Theorem \ref{4.1}. 
\begin{proof}[Proof of Theorem \ref{4.1}]
	Let $\mathcal{C}$ be a category whose objects are all trees and whose morphisms are pairs of maps $(s,i)$ between trees that satisfies conditions (1)-(3). By condition (3) there are trees $S$, $T$ and $(s,i)\in \text{Hom}_{\mathcal{C}}(S,T)$ so that $i\neq i_s$. Fix $x\in S$ so that $i_s(x)\neq i(x)$. We can assume that $i_s(x)$ has at least two immediate successors in $T$. Indeed, if not apply Proposition \ref{1.1} to $T$ to obtain a new tree $T'$. There is a rigid surjection $t\colon T'\to T$ so that $i_t(i_s(x))$ has at least two immediate successors.  Note that $i_t\circ i_s(x)\neq i_t\circ i(x)$ since $i_t$ is an injection. By condition (2), $(t,i_t)\in \text{Hom}_{\mathcal{C}}(T,T')$, so by composing with $(t,i_t)$ we assume that $i_s(x)$ has at least two immediate successors. Given this assumption we show that for any $V\in\text{Ob}(\mathcal{C})$, $V\not\to (T)^{S}_{2}$ in $\mathcal{C}$.
	
	Fix a tree $V$ and define a coloring $\chi\colon\text{Hom}_{\mathcal{C}}(S,V)\to 2$ by $$\chi(v,k)=\left\{ \begin{array}{ll}
		0 & \text{if }i_v(x)=k(x)\\
		1 & \text{if }i_v(x)\neq k(x)
		\end{array}
	 \right.$$
	   It is enough to show that for any connection $(u,j)\in\text{Hom}_{\mathcal{C}}(T',V)$,  $\chi((u,j)\circ (s,i_s))=0$  and $\chi((u,j)\circ (s,i))=1$. By Lemma \ref{2.16} we get that $i_u(i_s(x))=j(i_s(x))$, hence \\$\chi((u,j)\circ (s,i_s))=0$. It remains to show that $i_u(i_s(x)\neq j(i(x))$. To do so we prove that $$i_u(i_s(x))<_V i_u(i(x))\leq_V j(i(x))$$
	 For the first part of the inequality we note that $i_s(x)\sqsubseteq_T i(x)$ by the definition of $i_s$. This implies $i_s(x)\leq_T i(x)$ and $i_s(x)<i(x)$ since $i_s(x)\neq i(x)$. Because $i_u$ is an increasing injection, $i_u(i_s(x))<_V i_u(i(x))$. The last inequality follows from the definition of $i_u$ since for any $y\in T$, $i_u(y)\sqsubseteq_V j(y)$ (hence $i_u(y)\leq_T j(y))$.
\end{proof}
We finish this section proving Lemma 3.3 which gives a lower bound on the Ramsey degree for connections.
\begin{proof}[Proof of Lemma \ref{2.17}]
		Fix a tree $S$ and let $$A=\{x\in S\colon x\text{ is not the root of }S\text{ and }x\text{ has at most one successor in }S\}$$ We apply Proposition \ref{1.1} to $S$ to obtain a tree $T$. For an arbitrary tree $V$, we define a coloring $\chi\colon\text{Hom}_{Conn_T}(S,V)\to \mathcal{P}(A)$ by $$\chi(v,k)=\{a\in A\colon i_v(A)\neq k(A)\}$$ It is enough to show that for any $(t,j)\in \text{Hom}_{Conn_T}(T,V)$ and $B\subseteq A$,\\ $\chi((t,j)\circ (s,i_B))=B$. Here $(s,i_B)$ is the connection given by Proposition 4.2 where for all $a\in A$, $i_s(a)\neq i_B(a)$ if and only if $a\in B$.
		
		 Fix $(t,j)\in\text{Hom}_{Conn_T}(T,V)$ and $B\subseteq A$. We show $B\subseteq \chi(t,j)\circ (s,i_B)$, for this end let $b\in B$. Note that by Proposition 4.2, $i_s(b)\neq i_B(b)$, hence $i_s(b)<_T i_B(b)$ since $i_s$ is the smallest partial inverse of $s$. Because $j$ is an embedding, we have that $j(i_s(b))<_V j(i_B(b))$. By the definition of $i_t$ we have $i_t(i_s(b))\leq_V j(i_s(b))$. Thus $i_t(i_s(b))<_V j(i_B(i))$. This means that $b\in \chi((t,j)\circ (s,i_B))$, since we have proven that $i_t(i_s(b))\neq j(i_B(i))$.  Thus we obtain the first inclusion that $B\subseteq \chi((t,j)\circ (s,i_B))$.
		 
		  For the other inclusion we will prove the contrapositive. So we need to prove that if $a\in A\backslash B$, then $j\circ i_B(a)=i_t\circ i_s(a)$. By Proposition 4.2 $i_s(a)=i_B(a)$ and $i_s(a)$ has at least two successors in $T$. In this case Lemma \ref{2.16} states that $j(i_s(a))=i_t(i_s(a))$. Hence, $i_t\circ i_s(a)=j\circ i_{B}(a)$, which completes the proof.
\end{proof}
 In the next section we will prove that the lower bound given in Lemma \ref{2.17} is in fact an upper bound. 
\section{Upper bounds on Ramsey Degrees}
In this section we focus on Theorem \ref{2.18}. The goal of the first part of this section is to prove Theorem \ref{2.18}. To do so we start by giving some background. First, we discuss a general framework create by Solecki in \cite{23} which allows us to prove a simpler statement instead of proving the entire Ramsey property. We then give details about the proof of Theorem \ref{2.13} which we will use in our proof of Theorem \ref{2.18}. In doing so we expand the notion of sealed rigid surjection in \cite{22} for connections. Then we prove of Theorem \ref{2.18} by breaking it into multiple steps. We continue to reduce the problem until it follows from the main lemma used to prove Theorem 2.4 in \cite{22}. In Section 5.4 we show that the Self-Dual Ramsey Theorem for linear orders follows from Theorem \ref{2.18}
	\subsection{Transferring the Ramsey Property and Condition (P)}
	We review the work of Solecki in \cite{23} that we will use to prove Theorem \ref{2.18}. The key definition of \cite{23} is that of condition (P). This concept is a version of the Ramsey property that is defined for functors rather than categories. Fulfilling condition (P) can be propagated by so called frank functors. We will define condition (P), frank functors, and then we state facts from \cite{23} which explain how frank functors transfer fulfilling condition (P). We also describe a weaker condition called (FP) and state when fulfilling (FP) implies fulfilling condition (P).
	\begin{definition}
		Given categories $\mathcal{C},\mathcal{D}$, a functor $\delta\colon \mathcal{C}\to \mathcal{D}$, and $A,B\in \textnormal{Ob}(\mathcal{C})$. If \\$f'\in \delta(\textnormal{Hom}_{\mathcal{C}}(A,B))$, let $$\textnormal{Hom}_{C}^{f'}(A,B)=\{f\in \textnormal{Hom}_{\mathcal{C}}(A,B)\colon \delta(f)=f'\}.$$
		We say that $\delta$ \textbf{fulfills (P) at $A,B$}, if for all $r>0$ there is a $C\in\textnormal{Ob}(\mathcal{C})$, so that for each $r$-coloring $\chi\colon \textnormal{Hom}_{\mathcal{C}}(A,B)\to r$, there is a $g\in \textnormal{Hom}_{\mathcal{C}}(B,C)$ so that for all morhpisms $f'\in \delta(\textnormal{Hom}_{\mathcal{C}}(A,B))$, $g\circ \textnormal{Hom}_{C}^{f'}(A,B)$ is $\chi$-monochromatic.
		
		We say that $\delta$ \textbf{fulfills (P) at $A$} if $\delta$ fulfills (P) at $A,B$ for all $B\in\textnormal{Ob}(\mathcal{C})$.
	\end{definition}
	 We can view condition (P) as a generalization of the Ramsey property. Indeed, having the Ramsey property is a special case of fulfilling condition (P). Let $\mathcal{C}$ be a category and $\gamma\colon \mathcal{C}\to \{*\}$ is the unique functor from $\mathcal{C}$ to the category with one object. In this case $\gamma$ fulfills (P) at $A,B\in\text{Ob}(\mathcal{C})$ if and only if for all $r>0$ there is a $C\in \text{Ob}(\mathcal{C})$ so that $C\to(B)^{A}_{r}$.  This follows from  there only being one $f'\in\text{Hom}(*)$, hence $\text{Hom}_{\mathcal{C}}^{f'}(A,B)=\text{Hom}_{\mathcal{C}}(A,B)$. 
	
	Condition (P) gives an upper bound on Ramsey Degrees. Note that if $\delta\colon \mathcal{C}\to \mathcal{D}$ and $\delta$ satisfies (P) at $A,B\in \text{Ob}(\mathcal{C})$, then $rd(A,B)\leq |\delta(\text{Hom}_{\mathcal{C}}(A,B))|$. Next we will consider certain functors which transfer condition (P). 
	\begin{definition}
		Let $\mathcal{C},\mathcal{D}$ be categories and $A\in \text{Ob}(\mathcal{C})$, a functor $\delta\colon \mathcal{C}\to \mathcal{D}$ is called \textbf{frank at $A$} if for all  $D\in\textnormal{Ob}(\mathcal{D})$ there is a $B\in\textnormal{Ob}(\mathcal{C})$ so that $\delta(B)=D$ and $$\delta(\textnormal{Hom}_{\mathcal{C}}(A,B))=\textnormal{Hom}_{\mathcal{D}}(\delta(A),D).$$
		\end{definition}
	Note that by the definition of functor $\delta(\textnormal{Hom}_{\mathcal{C}}(A,B))\subseteq\textnormal{Hom}_{\mathcal{D}}(\delta(A),D)$
	so frank functors are the functors where $\delta(\textnormal{Hom}_{\mathcal{C}}(A,B))\supseteq\textnormal{Hom}_{\mathcal{D}}(\delta(A),D).$ 
	The following result explains how frank functors interact with condition (P).
	\begin{proposition}\label{2.4} Let $\mathcal{C},\mathcal{D},\mathcal{E}$ be categories $\delta\colon \mathcal{C}\to \mathcal{D}$ and $\gamma\colon \mathcal{D}\to \mathcal{E}$ be functors. For all $A,B\in\textnormal{Ob}(\mathcal{C})$ if $\delta$ fulfills (P) at $A$, $\delta$ is frank at $B$, and $\gamma$ fulfills (P) at $\delta(A),\delta(B)$ then $\delta\circ \gamma$ fulfills (P) at $A,B$.
	\end{proposition}
	This is a version of \cite[Theorem 3.1]{23}, which follows directly from the proof of \cite[Theorem 3.1]{23}, but has weaker hypothesis since we only assume that the functor is frank at $B$ instead of assuming it is frank at every object.
	
	We remark that if $\delta\colon \mathcal{C}\to \mathcal{D}$ has (P) at $A\in \text{Ob}(\mathcal{C})$ and is frank at $B\in \mathcal{C}$ then Proposition \ref{2.4} implies that for each $r>0$, there is a $C\in \text{Ob}(\mathcal{C})$ so that $C\to (B)^{A}_{r}$ in $\mathcal{C}$. Indeed, as noted above since $\mathcal{D}$ has the Ramsey property, the unique functor $\gamma\colon \mathcal{D}\to \{*\}$ has (P) at $\delta(A),\delta(B)$. By Proposition \label{2.4} $\gamma\circ \delta$ has (P) at $A,B$. However, $\gamma\circ \delta$ is the unique functor from $\mathcal{C}$ to $\{*\}$, hence $\gamma\circ \delta$ having (P) at $A,B$ means there is a $C\in \text{Ob}(\mathcal{C})$ so that $C\to (B)^{A}_{r}$ in $\mathcal{C}$.
	
		The other theorem from \cite{23} that we will be using involves a local version of condition (P) which is called condition (FP).
	\begin{definition}
		Fix categories $\mathcal{C},\mathcal{D}$, a functor $\delta\colon \mathcal{C}\to \mathcal{D}$, and $A,B\in \textnormal{Ob}(\mathcal{C})$.	We say that $\delta$ \textbf{fulfills (FP) at $A,B$} if for all $r>0$ and finite set $e\subseteq \delta(\textnormal{Hom}_{\mathcal{C}}(A,B))$ there are $C\in\textnormal{Ob}(\mathcal{C})$, $f'\in s$, and $g'\in \delta(\textnormal{Hom}(B,C))$, so that for each $r$-coloring $\chi\colon \textnormal{Hom}_{\mathcal{C}}(A,B)\to r$, there is a $g\in \textnormal{Hom}_{\mathcal{C}}(B,C)$ so that $g\circ \textnormal{Hom}_{C}^{f'}(A,B)$ is $\chi$-monochromatic and $g'\circ e=\delta(g)\circ e$ for all $e\in s$.
		
		We say that $\delta$ \textbf{fulfills (FP) at $A$} if $\delta$ fulfills (FP) at $A,B$ for all $B\in\textnormal{Ob}(\mathcal{C})$.
	\end{definition}
	In condition (P) we require the set $g\circ \textnormal{Hom}_{C}^{f'}(A,B)$ to be $\chi$-monochromatic for all $f'\in \delta(\textnormal{Hom}_{\mathcal{C}}(A,B)$ while in (FP) we fixed the $f'$ ahead of time. In exchange for this, condition (FP) requires us to restrict $\delta(g)$. The following proposition shows that for our purposes condition (FP) implies condition (P).
	\begin{proposition}[{\cite[Theorem 4.3]{23}}]\label{2.5}
		Let $\delta\colon \mathcal{C}\to \mathcal{D}$ be a functor, if $\delta$ fulfills (FP) at $A\in \textnormal{Ob}(\mathcal{C})$ then $\delta$ fulfills (P) at $A,B$ for each $B\in \textnormal{Ob}(\mathcal{C})$ with $\delta(\textnormal{Hom}_{\mathcal{C}}(A,B))$ finite.
	\end{proposition}
Due to Proposition 5.2, in this paper we will prove condition (FP) instead of proving condition (P).
\subsection{Strong Connections and Forests}
To prove Theorem \ref{2.18} we will need more background from the proof of Theorem \ref{2.12} in \cite{22}. We start by recapping the definition of sealed rigid surjections from \cite{22}, which we generalize to connections. This allows us to define a new category that we transfer the Ramsey property from. We end the section by stating a version of the key lemma used to prove Theorem \ref{2.12} in \cite{22} that we use in the proof of Theorem 3.4.

\begin{definition}
	A \textbf{leaf} of a tree $T$ is a $\sqsubseteq$ maximal element of $T$.
	
	We say that a rigid surjection $s\colon T\to S$ is a \textbf{sealed rigid surjcetion} if $i_s$ sends the $\leq_S$ largest leaf of $S$ to the $\leq_T$ largest leaf of $T$
\end{definition}

We can turn any rigid surjection into a sealed rigid by restricting its domain. We add new notation for such restrictions. If $T$ is a tree and $v\in T$, we let $T^v=\{y\in T\colon y\leq_T v\}$ we call such subtrees \textbf{initial subtrees} of $T$. For a rigid surjection $s\colon T\to S$ we let $s^v=s\restriction T^v$. If $s\colon T\to S$ is a rigid surjection and $v$ is the largest leaf of $S$, $s^{i_s(v)}$ is a sealed rigid surjection. Now we define an analog to sealed rigid surjections for connections.
\begin{definition}
	A connection $(s,i)\colon T\leftrightarrows S$ between trees $T$ and $S$ is \textbf{strong} if $i$ sends the largest $\leq_S$ leaf of $S$ to the largest $\leq_T$ leaf of $T$.
\end{definition}
Note that $s$ is a sealed rigid surjection if and only if $(s,i_s)$ is a strong connection. We will define a category of partial strong connections as follows. 

Let $PSC$ have objects be fine ordered trees and $$\textnormal{Hom}_{PSC}(S,T)=\{(s,i)\colon T^{v}\leftrightarrows S| \ v\in T \text{ and }(s,i)\text{ is a strong connection}\}.$$
Composition is defined by $(s,i)\circ(t,j)=(t\circ s^{i(v)},i^{v}\circ j)$ where $(t,j)\colon T^v\leftrightarrows S$. 

The fact that composition is well defined follows from the definition of connections. We state an upper bound on the Ramsey degrees for this category.
\begin{theorem}\label{2.20}
	If $S\in \textnormal{Ob}(PSC)$, let $$A=\{x\in S\colon x\text{ is not the root of }S\text{ and }x\text{ has at most one immediate successor in }S\}.$$
	Then for all $T\in \textnormal{Ob}(PSC)$ and $r>0$, there is a $V\in \textnormal{Ob}(PSC)$ so that for any coloring $\chi\colon \textnormal{Hom}_{PSC}(S,V)\to r$, there is a $(t,j)\in\textnormal{Hom}_{PSC}(T,V)$ such that for any $(s,i)\in\textnormal{Hom}_{PSC}(S,T)$, $\chi((t,j)\circ (s,i))$ depends only the set, $$B=\{x\in A\colon i_s(x)\neq i(x)\}.$$
	In particular, $rd(S)\leq |P(A)|$.
\end{theorem}
In Section 5.3 we prove Theorem 3.4 in two steps. We construct a frank functor with condition (P) from $Conn_T$ to $PSC$. By Proposition 5.1 such a functor reduces proving Theorem 3.4 to proving Theorem 5.3. For that proof we introduce a new functor which fulfills condition (FP).

In \cite{12} the author uses the Graham-Rotschild Theorem in \cite{8} to prove Theorem \ref{2.9} (the Ramsey statement for connections between linear orders). It is not enough to use Theorem \ref{2.12} to prove Theorem \ref{2.20}. However, we utilize the main lemma which proves Theorem \ref{2.12} in \cite{22}. To state this result we recall the notion of forests following the presentation from \cite{22}.
\begin{definition}
	A \textbf{forest} is a partial order $(T,\sqsubseteq$) where for all $x\in T$ the predecessors of $x$ are linearly ordered by $\sqsubseteq_T$.
\end{definition}
Note that any forest is a tree without the root. We use forests to construct new trees with the following method.

For any forest $T$, the tree $1\oplus T$ is constructed by adding a new element to $T$. We define  a tree order on $1\oplus T$ so that the new element is the root. We say that a forest is \textbf{ordered} if there is a $\leq_T$ which extends to a tree ordering on $1\oplus T$. 

We construct a new tree by adding multiple ordered forests $T_1,...,T_n$ to a tree $T$. If $x_1,...,x_n$ are distinct vertices of the tree $T$, then the ordered tree $$V=(T;x_1,...,x_n)\oplus (T_1,...,T_n)$$is defined with a  base set of the disjoint union of $T$ and all $T_i$. The tree relation $\sqsubseteq_V$ extends $\sqsubseteq_T$ and each $\sqsubseteq_{T_i}$. Additionally, if $x\sqsubseteq_T x_i$, then $x\sqsubset_V y$ for all $y\in T_i$. So we put each $T_i$ on top of $x_i$. The tree order $\leq_V$ extends $\leq_T$ and each $\leq_{T_i}$. Lastly, every $T_i$ is the final interval of the set $\{v\in V\colon x_i\sqsubseteq v\}$ in the order $\leq_V$.

Given these definitions we can state the main lemma from \cite{22} that we use in our proof of Theorem 3.4. We give a reformulation of {\cite[LP restatement 1]{22}} which is more applicable to our purposes.
\begin{lemma}\label{2.13}
	Let $S$ and $T$ be trees where $w$ is the second largest vertex of $S$ and $T$ is of the form $(T';x_1,\dots,x_n)\oplus(T_1,\dots,T_n)$. Fix a rigid surjection $s'\colon T'\to S^w$ and a number $r>0$. There is a tree of the form $V=(T';x_1,\dots,x_n)\oplus (V_1,\dots,V_n)$ where $V_1,\dots,V_n$ are ordered forests so that for any $r$-coloring $\chi$ of the sealed rigid surjections $u\colon V^x\to S$ which extend $s'$ and have $x\in V_1$, there is a rigid surjection $t\colon V\to T$ which extends the identity on $T'$ so that $$t\circ \{s\colon T^y\to S| \ s \text{ is a sealed rigid surjection which extends }s' \text{ and }y\in T_1\}$$
	is $\chi$ monochromatic.
\end{lemma}
In our proofs the $T_1,\dots,T_n$, $x_1,\dots,x_n$ and $T'$ are chosen so that any $s\colon T^y\to S$ which we consider has $y\in T_1$.
\subsection{Proof of Theorem \ref{2.18}}
In this section we prove the characterization of the self-dual Ramsey degrees for connections. We start in the first section by proving that Theorem 3.4 follows from Theorem \ref{2.20} (the statement for partial strong connections) by using a functor $\delta$. Then we prove Theorem \ref{2.20} using another functor $\partial$ in two parts. In the second section we define $\partial$ and show that if $\partial$ satisfies (FP), then Theorem \ref{2.20} holds. We prove that $\partial$ satisfies (FP) in the last section.
	\subsubsection{Transferring the Ramsey degrees from $PSC$ to $Conn_T$.}
	We prove that Theorem \ref{2.20} implies Theorem \ref{2.18}, using the following functor.
	\begin{definition}
		Let $\delta$ be the functor $\delta\colon Conn_T\to PSC$  defined by $\delta(T)=T$ and for a connection $(s,i)\colon T\leftrightarrows S$, let $\delta(s,i)=(s^{i(v)},i)$ where $v$ is the largest leaf of $S$.
	\end{definition}
	It is easy to check that $\delta$ is a functor. We remark that the functor $\delta$ does not change the embedding $i$, but changes the rigid surjection $s$. Next we show that this functor is frank and satisfies (FP).
	\begin{lemma}\label{4.3}
	The functor $\delta $ is frank.
	\end{lemma}
	\begin{proof}
		Let $T$ and $S$ be trees. We need to show that $\delta(\text{Hom}_{Conn_T}(S,T))=\text{Hom}_{PSC}(S,T)$, which proves that $\delta$ is frank. Pick $(s',i)\in \text{Hom}_{PSC}(S,T)$, we show that $(s',i)$ is in the image of $\delta$. By the definition of $PSC$, $s'\colon T^{w}\to S$ for some $w\in T$. Since $(s',i)$ is strong $w=i(v)$ where $v$ is the largest leaf in $S$. Define $s\colon T\to S$ by $s(y)=s'(y)$ for all $y\in T^{w}$ and $s(y)=x_0$ where $x_0$ is the root of $S$, for all $y>_T w$. Note that $i_{s'}=i_{s}$ by the definition of $s$. Hence, $s$ is a rigid surjection and $(s,i)$ is a connection. Since $\delta(s,i)=(s',i)$ we have completed the proof.
	\end{proof}
	To prove that Theorem \ref{2.20} implies our categorization of the the Self-Dual Ramsey degrees for trees it remains to show the following.
	\begin{lemma}\label{4.4}
		The functor $\delta$ satisfies condition (FP).
	\end{lemma}
	\begin{proof}
		Let $S$ and $T$ be trees, $e\subseteq \text{Hom}_{PSC}(S,T)$, and $r>0$. Pick $(s,i)\in e$ so that for all $(s',i')\in e$, $i'(w)\leq_T i(w)$ where $w$ is the largest leaf in $S$. We will construct a tree $V$ so that $T^{i(w)}\subseteq V$. We claim it is enough to show that for any coloring $$\chi\colon \text{Hom}_{Conn_T}(S,V)\to r$$ there is a $(v,j)\in \text{Hom}_{Conn_T}(T,V)$ so that $v\restriction T^{i(w)}=\text{Id}_{T^{i(w)}}$, $j^{i(w)} =\text{Id}_{T^{i(w)}}$ and $(v,j)\circ \text{Hom}_{Conn_T}^{(s,i)}(S,T)$ is monochromatic. By the definition of (FP) this means we need to find a $(v',j')\in \in \text{Hom}_{PSC}(T,V)$ so that $$\delta(v,j)\circ (s',i')=(v',j')\circ (s',i')$$ for all $(s',i')$ in $e$. Fix any $(v',j')\in \text{Hom}_{PSC}(T,V)$, so that $v'\restriction T^{i(w)}=\text{Id}_{T^{i(w)}}$ and $j'^{i(w)}=\text{Id}_{T^{i(w)}}$, then for any $(s',i')\in e$, $$(v',j')\circ (s',i')=(s\circ t'^{j'(i'(w))},j'^{i'(w)}\circ i')=(s',i')$$
		since $i'(w)\leq_T i(w)$ and $(v',j')$ extends the identity on $T^{i(w)}$. Similarly for any $(v,j)\in \text{Hom}_{Conn_T}(V,T)$ so that $v\restriction T^{i(w)}=\text{Id}_{T^{i(w)}}$ and $j^{i(w)}=\text{Id}_{T^{i(w)}}$, $\delta(v,j)\circ (s',i')=(s',i')$ for all $(s',i')\in e$. Thus $$\delta(v,j)\circ (s',i')=(v',j')\circ (s',i')=(s',i')$$
		which completes the proof of claim.
		
		 Next we construct the appropriate tree $V$.
		Note that any $(t,j)$ with $\delta(t,j)=(s,i)$ must have injection $j=i$. Thus, we only consider the  coloring of rigid surections. Let $x_0,x_2,x_3,...,x_n=i(w)$ enumerate the predecessors of $i(w)$ in the $\sqsubseteq_T$ order. We define the trees $T_m$ by $$T_m=\{x\in T\colon i(w)<_T x, \ x_m\sqsubseteq x, \text{ and if }i<n\text{ then }x_{m+1}\not\sqsubseteq x\}$$
		with the inherited relations. Note that $T=(T^{i(w)};x_0,...,x_n)\bigoplus(T_0,...,T_n)$. We can assume that $T_0$ is empty since every rigid surjection $t$ which extends $s$ must map all elements of $T_0$ to the root of $S$. Hence, the rigid surjection $t$ is completely determined by $t|_{T/T_0}$. We will apply Lemma 5.4 to new trees $T_*$ and $S_*$. Let $T_*$ be $T$ along with a new vertex $t_1$ added so that $t_1$ is the largest immediate successor of the root and similarly let $S_*$ be $S$ along with a new vertex $v$ that is the largest immediate successor of the root. Let $T_1=\{t_1\}$, so $T_*=(T^{i(w)};x_0,x_2,\dots,x_n)\oplus(T_1,T_2,\dots,T_n)$. We apply Lemma \ref{2.13} to $S_*$, $T_*$ and $s$ to obtain a tree $V_*=(T^{i(w)};x_0,x_2,\dots,x_n)\oplus(V_1,V_2,\dots,V_n)$. 
		We claim the tree $V=(T^{i(w)};x_2,\dots,x_n)\oplus(V_2,\dots,V_n)$ is as desired.
		
		Indeed, fix a coloring $\chi\colon \text{Hom}_{Conn_T}(V,S)\to r$. We define a new coloring $\chi_1$ on the rigid surjections $u\colon V\to S$ which extend $s$, by $\chi_1(u)=\chi(u,i)$ where we identify $i\colon S\to T$ with $i\colon S\to V$ (since $V$ contains a copy of $i(S)$). We once again define a new coloring, $\chi_2$ on the sealed partial rigid surjections $u\colon V_{*}^{x}\to S_*$ which extend $s$ and have $x\in V_1$, by $\chi_2(u)=\chi_1(u|V)$. This is a well defined coloring, since by the definition of rigid surjection $u|V$ will be a map from $V$ to $S$. Then by Lemma \ref{2.13} there is a rigid surjection $t\colon V_*\to T_*$ which extends the identity of $T^{i(w)}$ so that $$t\circ \{q\colon T_{*}^{y}\to S_*| \ q \text{ is a sealed rigid surjection which extends }s \text{ and }y\in T_1\}$$ is $\chi_2$ monochromatic. The map $t'\colon V\to T$ given by $t'=t|V$ is a well defined rigid surjection. Hence, by the definition of $\chi_2$, $$t'\circ \{q\colon T\to S| \ q \text{ is a sealed rigid surjection which extends }s \}$$ is $\chi_2$ monochromatic. Since $t'$ extends the identity on $T^{i(w)}$, so does $i_{t'}$. 
		Then $(t',i_{t'})\circ \text{Hom}_{Conn_T}^{(s,i)}(S,T)$ is $\chi$ monochromatic, by the definition of $\chi$, which completes the proof.
	\end{proof}

\subsubsection{The functor $\partial$}
Instead of defining a functor on the category $PSC$ we consider a local approach. In particular, we fix a subcategory $PSC_S$ for every tree $S$. 
\begin{definition}
	For each $S\in \text{Ob}(PSC)$, let the objects of $\bf{PSC_S}$ be all $T\in \text{Ob}(PSC)$ so that $\text{Hom}_{PSC}(S,T)\neq \emptyset$. For all $T\in \text{Ob}(PSC_S)\backslash S$, let $\text{Hom}_{PSC_S}(S,T)=\text{Hom}_{PSC}(S,T)$, and for all $V\in \text{Ob}(PSC_S)\backslash S$ we let $$\text{Hom}(PSC_S)(T,V)=\{(s,i)\in \text{Hom}(PSC)(T,V)\colon i=i_s\}.$$
\end{definition}
Note that for any $S\neq T\in \text{Ob}(PSC_S)$, if $rd_{PSC_S}(S,T)\leq k$, then $rd_{PSC}(S,T)\leq k$. 

The codomain of our functor will be $PSC_S$ with an additional bookkeeping mechanism.

\begin{definition}
	For each $S\in \text{Ob}(PSC)$ and $N>0$, we define a category $\bf PSC_{S\ast 2^N}$ with the same objects as $PSC_{S}$. If $T,V\in \text{Ob}(PSC_{S})\backslash S$, then the morphisms between them are the same as in $PSC_{S}$. For $T\in \text{Ob}(PSC_{S})\backslash S$ we have $$\text{Hom}_{PSC_{S\ast 2^N}}(S,T)=\text{Hom}_{PSC_S}(S,T)\times 2^N.$$ Composition is defined as in $PSC_S$ for trees $T\neq S$ ,$V\neq S$ and by $f\circ_{PSC_{S\ast 2^N}} (g,\vec{x})=(f\circ_{PSC_S}g,\vec{x})$.
\end{definition}

We define a functor $\partial$ which we show satisfies condition (P). This functor will remove largest element $v$ of $S$ and keep track of when $i(v)\neq i_s(v)$.

Let $S$ be a tree with multiple elements, let $v$ denote the largest $\leq_S$ element of $S$ and $w$ denote the second $\leq_S$ largest element of $S$.  We define $\partial\colon PSC_S\to PSC_{S^w\ast 2}$. On objects $\partial(T)=T$, if $T\in \text{Ob}_{PSC_S}\backslash S$ and $\partial(S)=S^{w}$. Let $\partial$ be the identity on $\text{Hom}_{PSC_S}(T,V)$ for all $T,V\in \text{Ob}(PSC_S)\backslash S$. For $T\in \text{Ob}(PSC_S)\backslash S$ and $(s,i)\in \text{Hom}_{PSC_S}(S,T)$ we define

$$\partial(s,i)=\left\{
\begin{array}{cc}
	(s^{i(w)},i^{w},1) & \text{ if }i(v)\neq i_s(v)\\
	(s^{i(w)},i^{w},0) & \text{ otherwise}\\
\end{array}\right.$$

So $\partial$ prunes the largest leaf $v$ from $S$, but remembers the information on whether $i(v)=i_s(v)$ or not. We show that $\partial$ is a functor. This proof uses that our definition of $\text{Hom}_{PSC_S}(T,V)$ only includes partial sealed rigid surjections for $T,V\in \text{Ob}(PSC_S)\backslash S$.
\begin{proposition}\label{4.5}
	For any tree $S$ with at least two elements, where $w$ is the second $\leq_S$ largest element, the map $\partial\colon PSC_S\to PSC_{S^{w}\ast 2}$ is a functor.
\end{proposition}
\begin{proof}
	It is clear that $\partial$ preserves identities and composition between $T,V,U\in\text{Ob}_{PSC_S}\backslash S$. So let $T,V\in \text{Ob}(PSC_S)\backslash S$, $(t,i_t)\in \text{Hom}_{PSC_S}(T,V)$, and $(s,i)\in\text{Hom}_{PSC_S}(S,T)$. We need to show that $(t,i_t)\circ(s^{i(w)},i^{w})=((s\circ t)^{i_t\circ i(w)},(i_t\circ i)^{w})$ and that $i_t\circ i(v)=i_{s\circ t}(v)$ if and only if $i(v)=i_s(v)$. The first condition follows from the definition of composition in $PSC$, while the second follows from $i_t$ being an injection and the fact that $i_{s\circ t}=i_t\circ i_s$.
\end{proof}
The goal of this section is to prove that $\partial$ fulfilling (FP) implies Theorem \ref{2.20}. To do this we iterate  $\partial$.

For a tree $S$ with at least two elements, where $w$ is the second $\leq_S$ largest vertex in $S$ and all $N>0$, we define a functor $\partial_N\colon PSC_{S\ast 2^N}\to PSC_{S^{w}\ast 2^{N+1}}$. Let $\partial_N$ be the identity on objects and $\text{Hom}_{PSC_{S\ast 2^N}}(T,V)$ for all $T,V\in\text{Ob}(PSC_{S\ast 2^N})\backslash S$. For $(s,i,\vec{x})\in \text{Hom}_{PSC_{S\ast 2^N}}(S,T)$, define $\partial_N(s,i,\vec{x})=(\partial(s,i),\vec{x})$.

The proof that $\partial_N$ is a functor follows from the fact $\partial$ being a functor. Furthermore, it is easy to see that if $\partial$ satisfies (FP), then $\partial_N$ satisfies (FP). We use this fact to prove Theorem \ref{2.20} from the statement that $\partial$ fulfills (FP). Then in the next section, we present the proof that $\partial$ has (FP).

\begin{proof}[Proof of Theorem \ref{2.20} from $\partial$ fulfilling (FP)]
	Let $S$ be a tree, and $T\in\text{Ob}(PSC_{S})\backslash S$.
	We claim that if $S=1$ where $1$ is the unique tree with a single element $|\text{Hom}_{PSC_1}(1,T)|=1$. Indeed, the domain of $s$ for a pair $(s,i)\in\text{Hom}_{PSC_1}(1,T)$ must be the root of $T$, since $i$ maps the root of $1$ to the root of $T$ and $(s,i)$ is strong. So we can assume that $|S|>1$ and let $n$ denote $|S|$.
	
	We define the functor  $$\Delta=\partial_{n-2}\circ\cdots\circ\partial_1\circ\partial$$ we claim that $\Delta$ fulfills $(P)$ at $S$.  To prove this, we first note that $\partial$ is sufficiently frank. For each $ T\in \text{Ob}(PSC_{S\ast 2^N})\backslash S$, $\partial_n$ is frank at $T$, since it is the identity on $\text{Hom}_{PSC_{S\ast 2^N}}(V,T)$ for $V\in \text{Ob}(PSC_{S\ast 2^N})\backslash S$. So since we assume that each $\partial_n$ fulfills (FP), Proposition \ref{2.4} and Proposition \ref{2.5} imply that $\Delta$ fulfills (P) at $S$. Thus $rd(S,T)\leq |\Delta(\text{Hom}_{PSC_S}(S,T))|$. To prove Theorem 5.3 it remains to show that for any $(s,i)\in\text{Hom}_{PSC_S}(S,T)$, $\Delta(s,i)$ depends on the set $$B=\{x\in A\colon i_s(x)\neq i(x)\}$$
	
	 Let $x_0,\dots,x_{n-1}$ enumerate $S$ in $\leq$ order. For any $(s,i)\in \text{Hom}_{PSC_{S}}(S,T)$, $\Delta(s,i)=(t,j,k_1,\dots,k_{n-1})$ where $(t,j)\in \text{Hom}_{PSC_1}(1,T)$ and  for all $N<n$, $k_N\in 2$ with $k_N=1$ if and only if $i(x_N)\neq i_s(x_N)$. As noted above $|\text{Hom}_{PSC_1}(1,T)|=1$, hence we identify $\Delta(s,i)$ with the sequence $(k_1,\dots, k_{n-1})$. Thus, $\Delta(s,i)$ is the indicator function of the set $B=\{x_N\in S\colon k_N=1\}$. By Lemma \ref{2.16} if $x_N$ has at least two immediate successors, then $i(x_N)=i_s(x_N)$ so $x_N\notin B$. Thus $B\subseteq A$, which proves $$B=\{x\in A\colon i_s(x)\neq i(x)\}.$$ 
\end{proof}
\subsubsection{Proof that $\partial$ fulfills (FP)}
In this section we simplify the statement that $\partial$ satisfies (FP) until it follows from Proposition \ref{2.4} and Lemma \ref{2.13}. We start by showing that $\partial$ fulfilling (FP) follows from two types of categories having the Ramsey property. Then we prove that the first class of categories have the Ramsey property by Lemma \ref{2.13}. For the second class we use Proposition \ref{2.4} to reduce this case to the first class of categories. We start with a reformulation of $\partial$ having (FP). This lemma is similar to the first claim in the proof of Lemma \ref{4.4}.
\begin{lemma}\label{4.6}
	Let $S$ be a tree with at least two elements, $w$ be the second largest $\leq_S$ element of $S$, and $T\in \text{Ob}(PSC_S)\backslash S$. If for all $r>0$ and $(s',i',a)\in \partial(\textnormal{Hom}_{PSC_{S}}(S,T))$ there is a tree $V$ so that $T^{i'(w)}$ is an initial subtree of $V$ and for any coloring $\chi\colon \rm{Hom}_{PSC_S}(S,V)\to r$, we can find a $(t,i_t)\in \rm{Hom}_{PSC_S}(T,V)$ so that $(t,i_t)\circ \rm{Hom}_{PSC_{S}}^{(s',i',a)}(S,T)$ is monochromatic and $(t^{i(w)},i_{t}^{i(w)})=\rm{Id}_{T^{i(w)}}$, then $\partial$ satisfies (FP) at $S,T$.
\end{lemma}
\begin{proof}
	Fix $S$ a tree, $T\in\text{Ob}(PSC_S)\backslash S$, $r>0$, and $e\subset \partial(\text{Hom}_{PSC_S}(S,T))$. We need to show that there are $V\in\textnormal{Ob}(PSC_S)$, $(s,i,a)\in e$ and $(u,i_u)\in \text{Hom}_{PSC_S}(T,V)$, so that for each $r$-coloring $\chi\colon \textnormal{Hom}_{PSC_S}(S,V)\to r$, there is a $(t,i_t)\in \textnormal{Hom}_{PSC_S}(T,V)$ for which \begin{center}$(t,i_t)\circ \rm{Hom}_{PSC_{S}}^{(s',i',a)}(S,T)$ is monochromatic and $(u,i_u)\circ (s',i',a')=(t,i_t)\circ (s',i',a')$\end{center} for all $(s',i',a')\in e$. Fix $(s,i,a)\in e$ so that $i(w)\geq_T i'(w)$ for all $(s',i',a')\in e$. Let $V$ be given by the assumption and $(u,i_u)$ be any rigid surjection so that\\ $(u^{i(w)},i_{u}^{i(w)})=\rm{Id}_{T^{i(w)}}$. By our hypothesis, it is enough to show that for any $(t,i_t)\in \text{Hom}_{PSC_S}(V,T)$ with $(t^{i(w)},i_{t}^{i(w)})=\rm{Id}_{T^{i(w)}}$,  \begin{center} $(u,i_u)\circ (s',i',a')=(t,i_t)\circ (s',i',a')$ for all $(s',i',a')\in e$.\end{center} By the definition of composition in $PSC_{S^{w}\ast 2}$ we need to show that $(u^{i'(w)},i_{u}^{i'{w}})\circ (s',i')=(t^{i'(w)},i_{t}^{i'(w)})\circ (s',i')$, but since $i'(w)\leq_T i(w)$, $$(u^{i'(w)},i_{u}^{i'(w)})=(t^{i'(w)},i_{t}^{i'(w)})=\text{Id}_{T^{i'(w)}}.$$
\end{proof}
In lieu of Lemma \ref{4.6} we will define new subcategories. Fix $S$  a tree with at least two elements, $T\in\text{Ob}(PSC_S)\backslash S$, and $(s,i,a)\in \partial(\text{Hom}_{PSC_S}(S,T))$. The category $PSC_{(s,i,a)}$ is a subcategory of $PSC$ with objects $S,T$ and trees $V$ so that $T^{i(w)}$ is an initial subtree of $V$, where $w$ is the second largest $\leq_S$ element of $S$. We define\\ $$\text{Hom}_{PSC_{(s,i,a)}}(S,T)=\text{Hom}_{PSC_S}^{(s,i,a)}(S,T)$$ For $V\in \text{Ob}(PSC_{(s,i,a)})\backslash S$ we let,
$$\text{Hom}_{PSC_{(s,i,a)}}(T,V)=\{(t,i_t)\in \text{Hom}_{PSC_S}(T,V)\colon (t^{i(w)},i_{t}^{i(w)})=\rm{Id}_{T^{i(w)}}\}.$$
All other morphisms in $PSC_{(s,i,a)}$ are the same as in $PSC_S$.

Note that Lemma \ref{4.6} states that if for all trees $S$ with at least two elements, $T\in \text{Ob}(PSC_S)\backslash S$, $(s,i,a)\in \partial(\text{Hom}_{PSC_S}(S,T))$, and $r>0$ there is a $V\in \text{Ob}(PSC_{(s,i,a)})$ so that $V\to(T)^{S}_{r}$ then $\partial$ fulfills (FP) at $(S,T)$. We will split this into two cases, one where $a=0$ and the other where $a=1$. 
\begin{lemma}\label{4.7}
	For all trees $S$, $T\in \text{Ob}(PSC_S)\backslash S$, $(s',i',0)\in \partial(\text{Hom}_{PSC_S}(S,T))$, and $r>0$ there is a $V\in \text{Ob}(PSC_{(s',i',0)})$ so that $V\to (T)^{S}_{r}$.
\end{lemma}
\begin{proof}
	Fix a tree $S$, $T\in \text{Ob}(PSC_S)\backslash S$, $(s',i',0)\in \partial(\text{Hom}_{PSC_S}(S,T))$, and $r>0$. If $v$ and $w$ are the $\leq_S$ largest and second largest elements of $S$ respectively, then we let $x=v\wedge_S w$. 	 
	We rewrite $T$ as some $(T';y_1,\dots,y_n)\oplus(T_1,\dots,T_n)$ so that we can apply Lemma 5.4. Hence we construct forests $T_1,...,T_n$. Let $i'(x)=y_1,...,y_n=i'(w)$ enumerate the set $\{y\in V| i(x)\sqsubseteq_Ty\sqsubseteq_T i'(w)\}$. For each $y_m$, let $$T_m=\{z\in T| y_m\sqsubseteq_T z, \text{ }i'(x)< z\text{, and if }m\neq n,\text{ then }y_{m+1}\not\sqsubseteq_T z \}$$ These are forests with the inherited relations. We define the subtree  $$T'=T^{i'(w)}\cup\{y\in T| i'(x)<_T y\text{ and }i'(x)\not\sqsubseteq_t y\}$$ Note that $T=(T';y_1,\dots,y_n)\oplus(T_1,\dots,T_n)$ since this construction adds back all forests we removed from $T$ to obtain $T'$.  We claim that for any $(s,i)\in \text{Hom}_{PSC_{(s',i',0)}}(S,T)$, $i(v)\in T_1$. Indeed, since $i$ is an embedding which extends $i'$, $$i(v)\wedge_T i(w)=i(v)\wedge_T i'(w)=i(x)=i'(x)$$ thus $i'(x)\sqsubseteq_T i(v)$. Next we show that $y_2\not\sqsubseteq i(v)$ by contradiction. If $y_2\sqsubseteq_T i(v)$ then $y_2=i'(w)\wedge_T i(v)$ which contradicts $i(v)\wedge_T i'(w)=i'(x)$.
	
	 Now apply Lemma \ref{2.13} to the trees $S$, $(T^{i'(w)};y_1,\dots,y_n)\oplus (T_1,\dots,T_n)$, the rigid surjection $s'$ and the fixed $r>0$. Thus we obtain a tree $V=(T^{i'(w)};y_1,\dots,y_n)\oplus(V_1,\dots,V_n)$, we claim that $V'=(T';y_1,\dots,y_n)\oplus(V_1,\dots,V_n)$ satisfies the Ramsey property.
	
Fix a coloring $\chi\colon \text{Hom}_{PSC_{(s',i',0)}}(S,V')\to r$. Define a coloring $\chi'$ of sealed rigid surjections $u\colon V^z\to S$ which extend $s'$ and where $z\in V_1$ by $\chi'(u)=\chi(u,j)$ where $j$ extends $i'$ with $j(v)=i_u(v)=z$. This is a well defined coloring, since $(u,j)\in\text{Hom}_{PSC_{(s',i',0)}}(S,V')$. Lemma 5.4 states that there is a rigid surjection $t\colon V\to (T^{i'(w)};y_1,\dots,y_n)\oplus(T_1,\dots,T_n)$ which extends the identity on $T^{i'(w)}$, so that $$t\circ \{s\colon T^y\to S| \ s \text{ is a sealed rigid surjection which extends }s' \text{ and }y\in T_1\}$$ is $\chi'$ monochromatic. We can extend $t$ to a rigid surjection $t'\colon V'\to T$ by letting it be the identity on the set $\{y\in T| i'(x)<_T y\text{ and }i'(x)\not\sqsubseteq_t y\}$. Note that any $(s,i)\in \text{Hom}_{PSC_{(s',i',0)}}(S,T)$ has $i$ extending $i'$ with $i(v)=i_s(v)$. Thus each $(t,i_t)\circ (s,i)$ is of the form $(u,j)$ where $$u=t\circ \{s\colon T^y\to S| \ s \text{ is a sealed rigid surjection which extends }s' \text{ and }y\in T_1\}$$ while $j$ is given in the definition of $\chi'$. Thus, \begin{center}$(t',i_{t'})\circ\text{Hom}_{PSC_{(s',i',0)}}(S,T)$ is $\chi$ monochromatic. \end{center}
\end{proof}
It remains to prove a Ramsey statement for the case where $a=1$. Our proof of the case uses Propsition 5.1 to propagate Lemma 5.9. For this goal, we define a new functor $\gamma$. Fix $S$  a tree with at least two elements, $T\in\text{Ob}(PSC_S)\backslash S$, and $(s,i,1)\in \partial(\text{Hom}_{PSC_S}(S,T))$. The map $\gamma\colon PSC_{(s,i,1)}\to  PSC_{(s,i,0)}$ is defined as the identity on objects and all morphisms except in $\text{Hom}_{PSC_{(s,i,1)}}(S,T)$. If $(q,j)\in \text{Hom}_{PSC_{(s,i,1)}}(S,T)$, we let $\gamma(q,j)=(q',j')$ where $q'=q^{i_q(v)}$, while we define $j'$ by $j'^{w}=j^{w}$ and $j'(v)=i_q(v)$. So the functor $\gamma$ restricts $q$ but doesn't change the surjection, while the injection is modified so that $j'(v)=i_q(v)$. The map $\gamma$ is a well defined functor, it is also clearly frank at $T$. By Proposition \ref{2.4} and Lemma \ref{4.7} it is enough to show that $\gamma$ fulfills condition (FP).
\begin{lemma}\label{4.8}
The functor $\gamma$ has (FP).
\end{lemma}
\begin{proof}	Fix $S$ a tree, $T\in\text{Ob}(PSC_S)\backslash S$, $r>0$, $(s,i,1)\in \partial(\text{Hom}_{PSC_S}(S,T))$, and\\ $e\subseteq \gamma(\text{Hom}_{PSC_{(s,i,1)}}(S,T))$. Let $v$ be the largest $\leq_S$ element of $S$. We need to show that there are $V\in\textnormal{Ob}(PSC_{(s,i,0)})$, $(q,j)\in e$ and $(u,i_u)\in \text{Hom}_{PSC_{(s,i,0)}}(T,V)$, so that for each $r$-coloring $\chi\colon \textnormal{Hom}_{PSC_S}(S,V)\to r$, there is a $(t,i_t)\in \textnormal{Hom}_{PSC_S}(T,V)$ for which\begin{center} $(t,i_t)\circ \rm{Hom}_{PSC_{S}}^{(q,j)}(S,T)$ is monochromatic and $(u,i_u)\circ (q,j)=(t,i_t)\circ (q',j')$\end{center} for all $(q',j')\in e$. Fix $(q,j)\in e$ such that $j(v)\geq_T j'(v)$ for all $(q',j')\in e$, where $v$ is the largest $\leq_S$ element of $S$.
	
We claim that if there is a tree $V$ so that $T^{j(v)}$ is an initial subtree of $V$ and for any coloring $\chi\colon \rm{Hom}_{PSC_{(s,i,1)}}(S,V)\to r$, there is a $(t,i_t)\in \rm{Hom}_{PSC_{(s,i,1)}}(T,V)$ for which $(t,i_t)\circ \rm{Hom}_{PSC_{S}}^{(q,j)}(S,T)$ is monochromatic and $(t^{j(v)},i_{t}^{j(v)})=\rm{Id}_{T^{j(v)}}$, then $\gamma$ satisfies (FP) at $S,T$.

 Let $V$ be the tree that is given by the assumption and $(u,i_u)$ be any function which satisfies $(u^{j(v)},i_{u}^{j(v)})=\rm{Id}_{T^{j(v)}}$. It is enough to show that for any $(t,i_t)\in \text{Hom}_{PSC_S}(V,T)$ so that $(t^{j(v)},i_{t}^{j(v)})=\text{Id}_{T^{j(v)}}$, $$(u,i_u)\circ (q,j)=(t,i_t)\circ (q',j')$$ for all $(q',j')\in e$. By the definition of composition in $PSC_{(s,i,0)}$ we need to show that $(u^{j'(v)},i_{u}^{j'(v)})\circ (q',j')=(t^{j'(v)},i_{t}^{j'(v)})\circ (q',j')$, but since $j'(v)\leq_T j(v)$, $$(u^{j'(v)},i_{u}^{j'(v)})=(t^{j'(v)},i_{t}^{j'(v)})=\text{Id}_{T^{j'(v)}}.$$
Hence we have proven the claim.

It remains to construct the tree $V$. We start by defining a tree $S_+=(S;v)\oplus 1$. This tree is obtained by adding one node called $+$ that is $\sqsubseteq_{S_+}v$. By Lemma \ref{4.7} there is a tree $V$, so that $V\to (T)^{S_+}_{r}$ in $PSC_{(q,j,0)}$.  We will show that this $V$ is as desired by coding elements of $\text{Hom}_{PSC_{(q,j,0)}}(S_+,V)$ as elements of $\text{Hom}_{PSC_{(s,i,1)}}(S,V)$.
	
Fix a coloring $\chi\colon\text{Hom}_{PSC_{(s,i,1)}}(S,V)\to r$.
	 We define a new coloring$$\chi'\colon \text{Hom}_{PSC_{(q,j,0)}}(S_+,V)\to r$$ by letting $\chi'(u,k)=\chi(u',k')$ where $(u',k')$ is constructed as follows.  Let $u'(y)=u(y)$ whenever $u(y)\in S$ and $u'(i_u(+))=v$. Similarly, let $k'(z)=k(z)$ for all $z\in S\backslash v$ and let $k'(v)=k(+)$. So $(u',k')$ extends $(u,k)$, but turns the vertex $+$ into a $k'(v)$ which is not $i_{u'}(v)$. Indeed, $$i_{u'}(v)<_T i_u(+)\leq_T k(+)=k'(v)$$It is easy to check that $(u',k')$ is a connection. Since $(u',k')$ extends $(q,j)$ on $S^{w}$, $\partial(u',k')=(s,i,1)$, hence $u',k'\in \text{Hom}_{PSC_{(s,i,1)}}(S,V)$. Therefore, $\chi'$ is a well defined coloring. By Lemma 5.9, there exists a $(t,i_t)\in\text{Hom}_{PSC_{(q,j,0)}}(T,V)$ so that \begin{center} $(t,i_t)\circ\text{Hom}_{PSC_{(q,j,0)}}(S_+,T)$ is $\chi'$-monochromatic.\end{center} It remains to show that for each \begin{center} $(p',l')\in \text{Hom}_{PSC_{(s,i,1)}}^{(q,j)}(S,T)$, $(t,i_t)\circ (p',l')$ is of the form $(u',k')$\end{center} where $(u,k)=(t,i_t)\circ (p,l)$ and $(p,l)\in\text{Hom}_{PSC_{(q,j,0)}}(S_+,T)$. Indeed, given this claim each morphism of the form $(t,i_t)\circ (p',l')\in \text{Hom}_{PSC_{(s,i,1)}}^{(q,j)}(S,T)$ must have the same color under $\chi$, namely the color of $(t,i_t)\circ (p,l)$ under $\chi'$.

Let $(p',l')\in\text{Hom}_{PSC_{(s,i,1)}}^{(q,j)}(S,T)$ and define  $p\colon T^{l'(v)}\to S_+$ as $p(y)=p'(y)$ for all $y\neq l'(v)$ and $p(l'(v))=+$. We let $l\colon S_+\to T^{l'(v)}$ be defined by $l=l'$ on $S^{w}$, $l(v)=i_{p'}(v)$, and $l(+)=l'(v)$.  We show that $$(p,l)\in\text{Hom}_{PSC_{(q,j,0)}}(S_+,T)$$Since $(p',l')$ is a connection it is easy to check that $(p,l)$ is a connection. Since $i_p(+)=l(+)$ and $\gamma(p',l')=(q,j)$, we have that $\partial(p,l)=(q,j,0)$, hence $$(p,l)\in\text{Hom}_{PSC_{(q,j,0)}}(S_+,T)$$
Define  $(u,k)=(t,i_t)\circ (p,l)$, we show  that $(u',k')=(t,i_t)\circ (p',l')$. The domain of $(t,i_t)\circ (p',l')=V^{i_t(l'(v))}=V^{i_t(l(+))}$ which is the domain of $u'$. We show that $u'=p'\circ t$. If $u(y)\in S$ then $p(t(y))\in S$ so by the definition of $p'$, $$p'(t(y))=p(t(y))=u(y)=u'(y)$$ Next $u'(i_u(+))=v$ and since $i_u=i_t(i_p)$, $$p'(t(i_u(+)))=p'(i_p(+))=v$$ Similarly, $$k'(z)=k(z)=i_t(l(z))=i_t(l'(z))$$ for all $z\in S^{w}$ and $$k'(v)=u(+)=i_t(l(+))=i_t(l'(v))$$ Thus $k'=i_t\circ l'$ as desired.
\end{proof}
We recap how we proved that $\partial$ satisfies condition (FP).  By Lemma \ref{4.6}, we needed to prove that for all trees $S$, $T\in \text{Ob}(PSC_S)\backslash S$, $(s,i,a)\in \partial(\text{Hom}_{PSC_S}(S,T))$, and $r>0$ there is a $V\in \text{Ob}(PSC_{(s,i,a)})$ so that $V\to (T)^{S}_{r}$. Lemma \ref{4.7} proved the case where $a=0$. For the case where $a=1$,  we proved that $\gamma$ satisfies (P). This completes the proof of Theorem 5.3  by Proposition \ref{2.4} and Lemma \ref{4.6}.
\subsection{Theorem 3.4 implies Self-Dual Ramsey Theorem for Linear Orders}
The proof of the Dual Ramsey Theorem for trees, implies the Graham-Rothschild Theorem as shown in \cite{22}. Similarly Theorem 3.4 extends the Self-Dual Ramsey Theorem for linear orders in \cite{12}. We show this in two steps, first we prove that connections which send the smallest element to the smallest element have the Ramsey Property. The reason for this condition is that we view increasing injections as embeddings between trees and embeddings for trees preserve the root. Then we prove the Self-Dual Ramsey Theorem for linear orders using this result.

\begin{proof}[Proof Theorem \ref{2.18} implies the Self-Dual Ramsey Theorem for linear orders]
	Let $Conn_{L*}$
	\\ be the category whose objects are linear orders and whose morphisms are connections $(s,i)$ so that $i(\text{min}(K))=\text{min}(L)$. We will show that this category has the Ramsey property. Fix linear orders $K,L$ which we view as trees (with $\sqsubset_K=\leq_K$ and the same for $L$) and fix a number $r>0$. By Proposition 4.2, there is a tree $T$ and a connection $(s,i)\colon T \leftrightarrows L$ so that $i(x)\neq i_s(x)$, for all $x\in L$ with at most one immediate successor that are not the root. Since $L$ is a linear order, this holds for all $x\in L$ that are not the minimum element. By Theorem 3.4 there is a tree $V$ so that for any coloring $\chi\colon \textnormal{Hom}_{Conn_T}(K,V)\to r$, there is a $(u,k)\in\textnormal{Hom}_{Conn_T}(T,V)$ such that for any $(t,j)\in\textnormal{Hom}_{Conn_T}(K,T)$, $\chi((u,k)\circ (t,j))$ depends only the set, $$B=\{x\in K\backslash\text{min}(K)\colon i_t(x)\neq j(x)\}.$$
	By construction, the set $B=K\backslash$min$(K)$ for all $(p,q)\in(s,i)\circ \text{Hom}_{Conn_T}(K,L)$. Hence we have that $V\to (L)^{K}_{r}$ in $Conn_T$. We view $V$ as a linear order with the $\leq_V$ order. Note $\text{Hom}_{Conn_T}(K,L)$ are the linear order connections $(s,i)$ where $i$ sends the minimum element of $K$ to the minimum element of $L$ ($s$ always maps the smallest elements to each other). We have thus obtained a Ramsey theorem for $\text{Conn}_{L*}$. Next we prove the full Self-Dual Ramsey Theorem for linear orders using a functor which fulfills condition (P).
	
	Let $Conn_L$ be the category whose objects are linear orders and morphisms are connections, while $(Fin,\leq)$ is the category whose objects are finite linear orders with morphisms being increasing injections. Fix a specific linear order $K$, we define a functor $\gamma\colon Conn_L\to (Fin,\leq)$ on objects by $\gamma(K)=1$ and $\gamma(L)=L$ for all $L\neq K$. On morphisms, if $(s,i)\in\text{Hom}_{Conn_L}(K,L)$ where $L\neq K$, let $\gamma(s,i)$ be the function that sends $1$ to $i(\text{min}(K))$. If $(s,i)\in \text{Hom}_{Conn_L}(L,M)$ where $L,M$ are both not $k$, let $\gamma(s,i)=i$. The map $\gamma$ is a functor and is frank at all $L\neq K$. Since $(Fin,leq)$ has the Ramsey Property by Ramsey's Theorem, it remains to show that $\gamma$ satisfies (FP) at $K$.
	
	Let $L\neq K$ be another linear order and $r>0$. For any $e\subset\gamma(\text{Hom}_{Conn_T}(K,L))$ let $i'$ be the injection with the largest $i'(1)$ which we call $n$. We define $$L'=\{x\in L\colon x\geq n\}$$ by the above there is a linear order $M'$ so that $M\to (L')^{k}_{r}$ in $Conn_{T}$. Let $M$ be the linear order which adds $n$ predecessors to $M'$ and $j\colon L\to M$ be an increasing injection which extends the identity on the first $n$ elements. We claim these $M,j'$ witness that $\gamma$ satisfies (FP). 
	
	Fix a coloring $\chi\colon \text{Hom}_{Conn_L}(K,M)\to r$, then define $\chi'\colon \text{Hom}_{Conn_{T}}(K,M')\to r$, by $\chi'(u,k)=\chi(u',k)$ where  $u'(y)=\text{min}(K)$ for all $y<n$, and $u'(y)=u(y)$ otherwise. Then there is a $j\in\text{Hom}_{Conn_{T}}(L',M')$ so that \begin{center}$j\circ \text{Hom}_{Conn_T}(K,L')$ is $\chi'$ monochromatic.\end{center} Let $(t',j')\colon M\leftrightarrows L$, extend $(t,j)$ so that $t'(x)=x$ and $j'(x)=x$ for all $x<n$. By construction $(t',j')\circ \text{Hom}_{Conn_L}^{i'}(K,L)$ is $\chi$ monochromatic, and $j'\circ i=p\circ i$ for all $i\in e$.
\end{proof}
	\bibliographystyle{plain}
		\bibliography{reference}
	\end{document}